\documentclass[11pt]{amsart}%
\usepackage{amscd,amsmath,latexsym,amsthm,amsfonts,amssymb,graphicx,color,geometry,hyperref}
\usepackage[utf8]{inputenc}
\usepackage{amsmath}
\usepackage{amsfonts}
\usepackage{amssymb}
\usepackage{graphicx}%
\providecommand{\U}[1]{\protect\rule{.1in}{.1in}}
\newtheorem{thm}{Theorem}[section]
\newtheorem{cor}[thm]{Corollary}

\newtheorem{lem}[thm]{Lemma}
\newtheorem{pps}[thm]{Proposition}
\newtheorem{conj}[thm]{Problem}
\theoremstyle{definition}
\newtheorem{exm}[thm]{Example}
\newtheorem{rem}[thm]{Remark}
\newtheorem*{thm*}{Theorem}
\newtheorem*{pps*}{Proposition}
\newtheorem*{lem*}{Lemma}

\theoremstyle{remark}
\numberwithin{equation}{section}

\def\<{\langle}
\def\>{\rangle}
\begin{document}
\title[Composition operators on Hardy-Smirnov spaces ]{Composition operators on Hardy-Smirnov spaces }
\author{\textsc{V.V. F\'{a}varo \, P.V. Hai, D.M. Pellegrino and O.R. Severiano }}
\address[V.V. F\'{a}varo]{Faculdade de Matem\'{a}tica, Universidade Federal de
Uberlândia, 38.400-902 - Uberl\^{a}ndia.}
\email{vvfavaro@ufu.br}
\address[P. V. Hai]{School of Applied Mathematics and Informatics, Hanoi University of Science and Technology, Vien Toan ung dung va Tin hoc, Dai hoc Bach khoa Hanoi, 1 Dai Co Viet, Hanoi, Vietnam.}
\email{hai.phamviet@hust.edu.vn}
\address[D.M. Pellegrino]{ Departamento de Matem\'{a}tica, Universidade Federal da
Para\'{\i}ba, 58.051-900 - Jo\~{a}o Pessoa, Brazil}
\email{daniel.pellegrino@academico.ufpb.br}

\address [O. R. Severiano]{ Programa Associado de P\'{o}s Gradua\c{c}\~{a}o em
Matem\'{a}tica Universidade Federal da Para\'{\i}ba/Universidade Federal
de Campina Grande, João Pessoa, Brazil.}
\email{osmar.rrseveriano@gmail.com}
\thanks{V.V. F\'{a}varo is supported by FAPEMIG Grant PPM-00217-18.}
\thanks{D.M. Pellegrino is supported by CNPq 307327/2017-5 and Grant 2019/0014 Paraiba
State Research Foundation (FAPESQ)}
\thanks{O.R. Severiano is postdoctoral fellowship at Programa Associado de Pós Graduação em Matemática UFPB/UFCG, and is
supported by INCTMat Grant 88887.613486/2021-00. }
\dedicatory{Dedicated to the memory of Glads Thais Flores.}
\begin{abstract}
We investigate composition operators $C_{\Phi}$ on the Hardy-Smirnov space
$H^{2}(\Omega)$ induced by analytic self-maps $\Phi$ of an open simply connected proper subset $\Omega$ of the complex plane.  When the
Riemann map $\tau:\mathbb{U}\rightarrow\Omega$ used to define the norm of
$H^{2}(\Omega)$ is a linear fractional transformation,
we characterize the composition operators whose adjoints are composition operators. As applications of this fact, we provide a new proof for the adjoint formula 
discovered by Gallardo-Guti\'{e}rrez and Montes-Rodr\'{i}guez and  we give a new approach to describe all Hermitian and unitary composition operators on $H^{2}(\Omega).$
Additionally, if the coefficients of $\tau$ are real, we exhibit concrete examples of
conjugations and describe the Hermitian and unitary composition
operators which are complex symmetric with respect to specific conjugations on
$H^{2}(\Omega).$ We finish this paper showing that if $\Omega$ is unbounded
and $\Phi$ is a non-automorphic self-map of $\Omega$ with a fixed point, then
$C_{\Phi}$ is never complex symmetric on $H^{2}(\Omega).$

\end{abstract}
\subjclass[2010]{47B33, 47B32, 47B99}
\keywords{Complex symmetry, Hardy-Smirnov space, composition operator.}
\maketitle
\tableofcontents

\section{Introduction}

The notion of composition operators appears in several branches of Mathematics
and its applications, as Dynamical Systems, Category Theory, Dirichlet Series
and Complex Analysis. For recent papers related to the subject in different
fields of Mathematics we refer to \cite{bs, Galindo, jge, le} and the
references therein. In this paper we will be interested in composition
operators in Complex Analysis.

Throughout this paper we will use the following notation: $\mathbb{C}$ denotes
the complex plane, $\mathbb{U}:=\{z\in\mathbb{C}:|z|<1\}$ is the open
unit disc, $\mathbb{C}_{+}:=\{z\in\mathbb{C}:\mathrm{Re}(z)>0\}$ is the open
right half-plane, $\Pi^{+}:=\{z\in\mathbb{C}:\mathrm{Im}(z)>0\}$ is the
open upper half-plane and $\mathbb{N}:=\{0,1,2\ldots\}.$

Hereafter $\Omega$ represents an open simply connected proper subset of the complex plane $\mathbb{C}$ and $\tau$ is a Riemann map, i.e., a biholomorphic
map that takes $\mathbb{U}$ onto $\Omega.$ For $0<r<1,$ let $\Gamma_{r}$ be
the curve in $\Omega$ defined by $\Gamma_{r}=\tau\left(  \left\{
|z|=r\right\}  \right) .$ The Hardy-Smirnov space $H^{2}(\Omega)$ is the set
of all analytic functions $f:\Omega\longrightarrow\mathbb{C}$ for which
\begin{align*}
\left\Vert f\right\Vert :=\left(  \sup_{0<r<1}\frac{1}{2\pi}%
\int_{\Gamma_{r}}|f(w)|^{2}|dw|\right)  ^{1/2}<\infty.
\end{align*}
It is well-known that the definition of $H^{2}(\Omega)$ does not
depend on the choice of $\tau$ and that $\left\Vert \; \cdot \;\right\Vert$ describes a
norm on $H^{2}(\Omega)$; moreover, different choices of $\tau$ provide
equivalent norms (see \cite[p. 63]{Shapiro-Smith}). Each Hardy-Smirnov space
$H^{2}(\Omega)$ is a Hilbert space, and $H^{2}(\Omega)$ is simply referred as
a Hardy space on $\Omega.$ If $\Omega=\mathbb{U}$ and $\tau$ is the identity
map, then $H^{2}(\Omega)$ reduces to the classical Hardy space of the open
unit disc.

If $\Phi$ is an analytic self-map of $\Omega,$ then $\Phi$ induces a linear
composition operator $C_{\Phi}$ on the space $H^{2}(\Omega)$ defined
by
\[
C_{\Phi}(f)=f\circ\Phi,\quad f\in H^{2}(\Omega).
\]
From now on we shall write $C_{\Phi}f$ instead of $C_{\Phi}(f)$ for short.
Operators of this type have been studied on a variety of spaces, with special
interest on the Hardy space $H^{2}(\mathbb{U})$ (see
\cite{Cowen-MacClauer,Elliot-Jury,Avendano-Rosenthal,Joel
Shapiro,Shapiro-Smith}). If $\Phi$ is an analytic self-map of $\mathbb{U}$,
then $C_{\Phi}$ is bounded on $H^{2}(\mathbb{U})$ (see \cite[Littlewood's
Theorem, p. 16]{Joel Shapiro}) but, for $\Omega\neq\mathbb{U}$, $C_{\Phi}$ is
not bounded in general. The main direction of research in this framework
relies on the comparison between the properties of $C_{\Phi}$ with those of
$\Phi.$ For instance, if $\Omega=\mathbb{C}_{+},$ Elliot and Jury
(\cite[Theorem 3.1]{Elliot-Jury}) established a boundedness criterion for
$C_{\Phi}$ in terms of angular derivatives. First, let us recall some
terminology. A sequence of points $z_{n}=\mathrm{Re}(z_{n})+i\mathrm{Im}%
(z_{n})$ in $\mathbb{C}_{+}$ is said to approach $\infty$ non-tangentially if
$\mathrm{Re}(z_{n})\longrightarrow\infty$ and the ratios $|\mathrm{Im}%
(z_{n})|/\mathrm{Re}(z_{n})$ are uniformly bounded. We say that a self-map
$\Phi$ of $\mathbb{C}_{+}$ fixes the infinity non-tangentially if $\Phi
(z_{n})\longrightarrow\infty$ whenever $z_{n}\longrightarrow\infty$
non-tangentially, and we denote $\Phi(\infty)=\infty.$ If $\Phi(\infty
)=\infty,$ we say that $\Phi$ has a finite angular derivative at $\infty$ if
the non-tangential limit $\Phi^{\prime}(\infty):=\displaystyle\lim
_{z\longrightarrow\infty}z/\Phi(z)$ exists. The result of Elliot and Jury
states that for an analytic self-map $\Phi$ of $\mathbb{C}_{+},$ the
composition operator $C_{\Phi}$ is bounded on $H^{2}(\mathbb{C}_{+})$ if, and
only if, $\Phi$ has a finite angular derivative at $\infty$. This example
illustrates that the boundedness of $C_{\Phi}$ on $H^{2}(\Omega)$ is strongly
influenced by the choice of $\Omega$.

Let $\mathcal{L}(\mathcal{H})$ denote the space of all bounded linear
operators on a separable complex Hilbert space $\mathcal{H}.$ Recall that $T\in \mathcal{L}(\mathcal{H})$ is called \textit{normal} if $TT^{\ast
}=T^{\ast}T,$ \textit{Hermitian} if $T=T^{\ast},$ \textit{unitary} if $TT^*=T^*T=Id,$
\textit{hyponormal} if $\left\|
T^{*}x\right\|  \leq\left\|  Tx\right\|  $ for all $x\in\mathcal{H},$ and
\textit{cohyponormal} if its adjoint $T^{*}$ is hyponormal. 
It is simple to check that the Hermitian
and unitary composition operators on $H^{2}(\mathbb{U})$ are precisely those
induced by $\Phi(z)=\lambda z$ with $\lambda\in\lbrack-1,1]$ and
$|\lambda|=1,$ respectively. For the case of normal operators (see
\cite[Theorem 5.1.15]{Avendano-Rosenthal}), composition operators on
$H^{2}(\mathbb{U})$ are induced by $\Phi(z)=\lambda z$ with $|\lambda|\leq1.$
The characterization of unitary composition operators on Hardy-Smirnov spaces
was given by Gunatillake, Jovovic and Smith in \cite[Theorem 3.2]{Jovic}:
$C_{\Phi}$ is a unitary bounded composition operator on $H^{2}(\Omega)$ if,
and only if, there are complex numbers $\lambda$ and $r$ with $|\lambda|=1$
such that $\Phi(z)=\lambda z+r$ is an automorphism of $\Omega.$ The case of
Hermitian composition operators on $H^{2}(\Omega)$ was studied by Gunatillake
in \cite[Theorem 3.3]{Gajath}, where it was shown that if $C_{\Phi}$ is a
Hermitian bounded composition operator, then $\Phi(z)=\lambda z+r$ for some
real number $\lambda$. However, this is not always sufficient to assure that
$C_{\Phi}$ is Hermitian.

 A \textit{conjugation} on $\mathcal{H}$ is a conjugate-linear operator
$C:\mathcal{H}\longrightarrow\mathcal{H}$ satisfying the following conditions:

\begin{enumerate}
\item $C$ is \textit{isometric}: $\left\langle Cf,Cg\right\rangle
=\left\langle g,f\right\rangle , \forall f,g\in\mathcal{H}.$

\item $C$ is \textit{involutive}: $C^{2}=Id.$
\end{enumerate}

According to \cite{Garc2}, $T\in\mathcal{L}(\mathcal{H})$ is $C$-\textit{symmetric} if
$CT=T^{\ast}C,$ and $T$ is \textit{complex symmetric} if there exists a conjugation $C$
with respect to which $T$ is $C$-symmetric. As can be seen in
\cite{Garc3,Garc4,Wogen1,Wogen2}, the class of complex symmetric operators
contains a large number of concrete examples including all normal operators,
binormal operators, Hankel operators, finite Toeplitz matrices, compressed
shift operators and the Volterra operators.

The investigation of complex symmetric composition operators on $H^{2}%
(\mathbb{U})$ was initiated by Garcia and Hammond in \cite{Garc1}. They
described all (weighted) composition operators that are $J$-symmetric, where
$J:H^{2}(\mathbb{U})\longrightarrow H^{2}(\mathbb{U})$ is the conjugation
given by
\[
(Jf)(z)=\overline{f(\overline{z})}%
\]
They have also shown that constant self-maps of $\mathbb{U}$ and involutive
disc automorphisms induce non-normal complex symmetric composition operators
on $H^{2}(\mathbb{U}).$ Since then, several authors have been interested in
the subject, describing complex symmetric composition operators and exhibiting
concrete conjugations for which $C_{\Phi}$ is complex symmetric (see
\cite{Wal, Hai, Han, Narayan}). Results on how the complex symmetry of
$C_{\Phi}$ affects the analytic self-map $\Phi$ were obtained in
\cite{Wal,Garc1}. More precisely, in \cite[Proposition p. 106]{Wal} the
authors prove that if $C_{\Phi}$ is complex symmetric on $H^{2}(\mathbb{U}),$
then $\Phi$ has a fixed point in $\mathbb{U}$ and, additionally, if $\Phi$ is
non-constant, then it is univalent (\cite[Proposition 2.5]{Garc1}). In
\cite[Theorem 2.5]{Maria}, Narayan, Sievewright and Tjani have shown that if
$\Phi$ is an analytic self-map of $\mathbb{U}$ which is not an elliptic
automorphism and $C_{\Phi}$ is complex symmetric on $H^{2}(\mathbb{U})$, then
$\Phi(0)=0$ or $\Phi$ is linear. They also characterized the non-automorphic
linear self-maps of $\mathbb{U}$ that induce complex symmetric composition
operators (see \cite[Theorem 3.1]{Maria}). In \cite{Osmar}, Noor and Severiano
characterized all linear fractional self-maps $\Phi$ of $\mathbb{C}_{+}$ that
induce complex symmetric composition operators $C_{\Phi}$ on $H^{2}%
(\mathbb{C}_{+}).$ In contrast with complex symmetric composition operators on
$H^{2}(\mathbb{U}),$ Hai and Severiano (\cite[Theorem 8.8]{Hai2}) showed that
if $\Phi$ is an analytic self-map of $\mathbb{C}_{+}$ with a fixed point in
$\mathbb{C}_{+}$ then $C_{\Phi}$ is never complex symmetric on $H^{2}%
(\mathbb{C}_{+})$. Since unitary and Hermitian bounded operators are complex
symmetric, \cite[Theorem 3.3]{Gajath} and \cite[Theorem 3.2]{Jovic} provide
natural ways to obtain examples of complex symmetric composition operators on
$H^{2}(\Omega).$ However, it is not immediate with respect to which
conjugations these operators are complex symmetric.

In this paper we are mainly interested in describing the unitary, Hermitian,
(partially) normal and (partially) complex symmetric composition operators on
$H^{2}(\Omega),$ whenever $\Omega$ is the range of a linear fractional Riemann
map $\tau$ from $\mathbb{U}.$ The paper is organized as follows. In Section
\ref{089}, we collect some general results that will be used in the subsequent
sections. In Section \ref{090}, considering that $\Omega$ is the range of a
linear fractional Riemann map $\tau$ from $\mathbb{U}$, we obtain an
expression for the reproducing kernels on $H^{2}(\Omega)$. Using these
expressions, in Section \ref{097}, we describe all composition operators whose
adjoints are also composition operators (Theorem \ref{038}) and we
characterize the Hermitian and unitary composition operators on $H^{2}%
(\Omega)$ (Theorems \ref{059} and \ref{060}). In particular, we show how the
numbers $\lambda$ and $r$ given in \cite[Theorem 3.3]{Gajath} and
\cite[Theorem 3.2]{Jovic} depend on the coefficients of $\tau.$ We also
provide concrete examples of normal composition operators on $H^{2}(\Omega). $
Section \ref{062} is devoted to the study of complex symmetric composition
operators. We exhibit concrete examples of conjugations on each $H^{2}%
(\Omega)$ (Propositions \ref{012} and \ref{098}) and, assuming that $\Omega$
is the range of a linear fractional Riemann map with real coefficients, we
provide concrete examples of complex symmetric operators (see Propositions
\ref{063} and \ref{091}). More precisely, in Propositions \ref{063} and
\ref{091} we explicit the conjugations with respect to which Hermitian and
unitary composition operators are complex symmetric. We finish the paper by
showing that if $\Omega$ is unbounded and $\Phi$ is an analytic
non-automorphic self-map of $\Omega$ with a fixed point of $\Omega$, then
$C_{\Phi}$ is not complex symmetric on $H^{2}(\Omega).$

\section{Background}

\label{089}

\subsection{Hardy-Smirnov space}

For $1\leq p <\infty,$ the classical Hardy space $H^{p}(\mathbb{U})$ is the
space of all analytic functions $f:\mathbb{U}\longrightarrow\mathbb{C}$ for
which
\begin{align*}
\left\|  f\right\|_p  : =\left(  \sup_{0< r<1}\frac{1}{2\pi}\int
_{0}^{2\pi}|f(re^{i\theta})|^{p}d\theta\right)  ^{1/p}<\infty.
\end{align*}
The norm $\left\| \; \cdot \; \right\|_p $ makes $H^p(\mathbb{U})$ into a Banach space. For $p=2,$ the inner product on $H^2(\mathbb{U})$ is defined by
\begin{align}
\label{096}\left\langle f,g\right\rangle_2 :=\sum_{n=0}^{\infty}\widehat
{f}(n)\overline{\widehat{g}(n)}, \quad f,g\in H^{2}(\mathbb{U})
\end{align}
where $(\widehat{f}(n))_{n\in\mathbb{N}}$ and $(\widehat{g}(n))_{n\in
\mathbb{N}}$ are the sequences of Maclaurin coefficients for $f$ and $g,$
respectively. Since 
\begin{align*}\left\| f\right\| _2^2=\sum_{n=0}^{\infty}|\widehat{f}(n)|^{2}, \quad f\in H^2(\mathbb{U})
\end{align*}
(see \cite[Theorem 1.1.12]{Avendano-Rosenthal}), it follows that the norm  $\left\| \; \cdot \; \right\|_2 $ makes $H^2(\mathbb{U})$ into a Hilbert space.

 The functions $\{K_{\alpha}:\alpha\in\mathbb{U}\}$ defined by
\begin{align*}
K_{\alpha}(z)=\frac{1}{1-\overline{\alpha}z}, \quad z\in\mathbb{U}%
\end{align*}
are called \textit{reproducing kernels} for $H^{2}(\mathbb{U}).$ These
functions have the fundamental property $f(\alpha)=\left\langle f,K_{\alpha
}\right\rangle $ for each $f\in H^{2}(\mathbb{U})$ and $\alpha\in\mathbb{U}.$

Henceforth we will use the symbol $\tau$ to denote the Riemann map
$\tau:\mathbb{U}\longrightarrow\Omega$ which induces the norm of $H^{2}%
(\Omega).$ The following result shows that the abstract aspect of
$H^{2}(\Omega)$ can be eased by an isometric isomorphism between $H^{2}%
(\Omega)$ and $H^{2}(\mathbb{U}).$

\begin{pps}
\label{03} Let $f$ be an analytic function on $\Omega.$ Then $f\in
H^{2}(\Omega)$ if and only if $(f\circ\tau)(\tau^{\prime})^{1/2}\in
H^{2}(\mathbb{U}).$ Moreover, the map $V$ given by $Vf=(f\circ\tau)
(\tau^{\prime})^{1/2}$ is an isometric isomorphism from $H^{2}(\Omega)$ onto
$H^{2}(\mathbb{U}).$
\end{pps}

\begin{proof}
See \cite[Corollary, p. 169]{Duren}.
\end{proof}
If $V$ is the isometric isomorphism of Proposition \ref{03}, then it is easy to check that $V^{-1}f=(\tau'\circ \tau^{-1})^{-1/2}(f\circ \tau^{-1})$ for each $f\in H^2(\mathbb{U}).$ Moreover, if $\left\langle
\cdot,\cdot\right\rangle $ is the inner product on $H^{2}(\mathbb{U})$ defined
in \eqref{096}, then
\[
\left\langle f,g\right\rangle :=\left\langle Vf,Vg\right\rangle_2
,\quad f,g\in H^{2}(\Omega)
\]
is an inner product on $H^{2}(\Omega)$ and $\left\langle f,f\right\rangle
=\left\Vert Vf\right\Vert ^{2}_2=\left\Vert f\right\Vert %
^{2}.$ Thus, $H^{2}(\Omega)$ is a Hilbert space endowed with the inner product
$\left\langle \cdot,\cdot\right\rangle.$

\subsection{Weighted composition operator $W_{\varphi}$}

Let $\varphi:\mathbb{U}\longrightarrow\mathbb{U}$ and $\psi:\mathbb{U}%
\longrightarrow\mathbb{C}$ be analytic functions. The weighted
composition operator $W_{\psi, \varphi}$ on $H^{2}(\mathbb{U})$ is defined
by
\begin{align*}
W_{\psi, \varphi}f=\psi\cdot f\circ\varphi, \quad f\in H^{2}(\mathbb{U}).
\end{align*}

In general, $W_{\psi,\varphi}$ is not bounded. If $\psi$ is bounded on
$\mathbb{U},$ then $W_{\psi,\varphi}$ is bounded on $H^{2}(\mathbb{U})$.
However $W_{\psi,\varphi}$ may still be bounded even if $\psi$ is not bounded
on $\mathbb{U}.$ The following lemma is well-known and shows how
$W_{\psi,\varphi}^{\ast}$ acts on reproducing kernels.

\begin{lem}
\label{05} If $W_{\psi,\varphi}$ is bounded on $H^{2}(\mathbb{U})$ and
$\alpha\in\mathbb{U}$ then
\begin{align}
\label{04}W_{\psi, \varphi}^{*}K_{\alpha}=\overline{{\psi}(\alpha)}%
K_{\varphi(\alpha)} .
\end{align}

\end{lem}

\begin{proof}
See \cite[Lemma 3.2]{Shapiro-Smith}.
\end{proof}
It is immediate from the lemma above that if $\psi\equiv 1$, then  $C_{\varphi}^{*}K_{\alpha}=K_{\varphi(\alpha)}$ for each $\alpha \in \mathbb{U}.$

Recall that two operators $T_{1}:\mathcal{H}_{1}\longrightarrow\mathcal{H}%
_{2}$ and $T_{2}:\mathcal{H}_{2}\longrightarrow\mathcal{H}_{1}$ are
\textit{similar} if there is an isomorphism $U:\mathcal{H}_{1}\longrightarrow
\mathcal{H}_{2}$ such that $UT_{1}U^{-1}=T_{2}.$ If $U$ is an isometry, we say
that $T_{1}$ and $T_{2}$ are \textit{isometrically similar}. Using Proposition
\ref{03} we see that each composition operator on $H^{2}(\Omega)$ is
isometrically similar to a weighted composition operator on $H^{2}%
(\mathbb{U}).$
In fact, a straight computation reveals that for each analytic self-map $\Phi$
of $\Omega$ we have
\[
VC_{\Phi}V^{-1}=W_{\left(  \frac{\tau^{\prime}}{\tau^{\prime}\circ\varphi
}\right)  ^{1/2},\varphi},
\]
where $\varphi=\tau^{-1}\circ\Phi\circ\tau.$ For simplicity, in this case, we
write $W_{\varphi}$ instead of $W_{\left(  \frac{\tau^{\prime}}{\tau^{\prime
}\circ\varphi}\right)  ^{1/2},\varphi}.$

We refer to \cite{Cowen-MacClauer, Duren, Avendano-Rosenthal,Joel Shapiro} and
\cite{Paul2, Shapiro-Smith} for more details on the Hardy space and on weighted composition operators, respectively.

\subsection{Complex symmetric operators}

Each complex symmetric operator satisfies the following \textit{spectral
symmetry property}: if $T\in\mathcal{L}(\mathcal{H})$ is $C$-symmetric and
$\lambda\in\mathbb{C},$ then
\[
f\in\mathrm{Ker}(T-\lambda I)\iff Cf\in\mathrm{Ker}(T^{\ast}-\overline
{\lambda}I).
\]

The next simple lemma, which will be stated for further reference, shows how
to obtain a new conjugation from a given conjugation.

\begin{lem}
\label{065} Let $C$ be a conjugation on $\mathcal{H}.$ If $U:\mathcal{H}%
_{1}\longrightarrow\mathcal{H}$ is an isometric isomorphism then $U^{-1}CU$ is
a conjugation on $\mathcal{H}_{1}.$
\end{lem}

In Section \ref{062}, we will make extensive use of the following well-known result.


\begin{pps}
Let $T\in\mathcal{L}(\mathcal{H})$ be $C$-symmetric. If $T$ is
isometrically similar to $T_1\in\mathcal{L}(\mathcal{H})$ via the
isometric isomorphism $U$, then $T_1$ is $U^{-1}CU$-symmetric.
\end{pps}

\begin{proof}
See \cite[p. 1291]{Garc2}.

\end{proof}

\section{Linear fractional Riemann maps}

\label{090}

A function $k\in H^{2}(\Omega)$ is a reproducing kernel for $H^{2}(\Omega)$ at
the point $u\in\Omega$ if $f(u)=\left\langle f,k\right\rangle $ for all $f\in
H^{2}(\Omega),$ and in this case, such reproducing kernel is denoted by
$k_{u}.$ The reproducing kernels on $H^{2}(\mathbb{U})$ are commonly used to
obtain properties of various bounded operators, such as: Hankel operators, Toeplitz operators
and weighted composition operators (see \cite{Cowen-Ko,Eva-Kumar, Garc1}). Generally,
this is done by testing the action of those bounded operators on the
reproducing kernels on $H^{2}(\mathbb{U}).$


This section is motivated by the observation done in the previous paragraph.
In this sense, we first characterize the functions that are the kernel
functions on $H^{2}(\Omega)$ (see Lemma \ref{024}), and then in Section
\ref{097} we use this characterization to study the composition operators on $H^2(\Omega).$

\subsection{Reproducing kernels}

A linear fractional Riemann map $\tau:\mathbb{U}\longrightarrow\Omega$ is a
Riemann map of the following form
\[
\tau(z)=\frac{az+b}{cz+d},\quad z\in\mathbb{U}.
\]
where $a,b,c$ and $d$ are complex numbers. Since $\tau$ is a Riemann map, it
is subject to the further condition $ad-bc\neq0.$ For instance, $\tau_{1}:$
$\mathbb{U}\rightarrow\mathbb{C}_{+} $ and $\tau_{2}:$ $\mathbb{U}%
\rightarrow\Pi^{+}$ given by%

\[
\tau_{1}(z)=\frac{1+z}{1-z}\quad\text{and}\quad\tau_{2}(z)=i\frac{1+z}{ 1-z}%
\]
are linear fractional Riemann maps.


As we see below, one of the advantages of studying $H^{2}(\Omega)$ where
$\Omega$ is a range of a linear fractional Riemann map $\tau$ is that the
reproducing kernels have a simple expression in terms of the coefficients of
$\tau.$

\begin{lem}
\label{024}
The reproducing kernels on $H^{2}(\Omega)$ are given by
\[
k_{u}=\overline{(\tau^{\prime}(\tau^{-1}(u)))}^{-1/2}V^{-1}%
K_{\tau^{-1}(u)},\quad u\in\Omega,
\]
where $V$ is the isometric isomorphism of Proposition \ref{03}. In particular,
if
\begin{align}\label{022}
\tau(z)=\frac{az+b}{cz+d},\quad z\in\mathbb{U}\text{,}%
\end{align}
then
\[
k_{u}(w)=\frac{|ad-bc|}{|a|^{2}-|b|^{2}+(b\overline{d}-\overline{c}%
a)\overline{u}+[(|c|^{2}-|d|^{2})\overline{u}+\overline{b}d-\overline{a}%
c]w},\quad u,w\in\Omega.
\]

\end{lem}

\begin{proof}
Since $\tau$ is a Riemann map, $\tau^{\prime}$ never vanishes in $\mathbb{U}.$
Hence, by Proposition \ref{03}, for each $u\in\Omega$ and $f\in H^{2}%
(\Omega),$ we obtain
\begin{align*}
f(u) &  =(\tau^{\prime}(\tau^{-1}(u)))^{-1/2}(f\circ\tau)(\tau^{-1}%
(u))(\tau^{\prime}(\tau^{-1}(u)))^{1/2}\\
&  =(\tau^{\prime}(\tau^{-1}(u)))^{-1/2}(Vf)(\tau^{-1}(u))\\
&  =\langle Vf,\overline{(\tau^{\prime}(\tau^{-1}(u)))}^{-1/2}K_{\tau^{-1}%
(u)}\rangle_2\\
&  =\langle f,\overline{(\tau^{\prime}(\tau^{-1}(u)))}^{-1/2}V^{-1}%
K_{\tau^{-1}(u)}\rangle,
\end{align*}
which shows that $\overline{(\tau^{\prime}(\tau^{-1}(u)))}^{-1/2}V^{-1}%
K_{\tau^{-1}(u)}$ is the reproducing kernel for $H^{2}(\Omega)$ at the point
$u.$ Now assume that $\tau$ is as in \eqref{022}. Then
\[
\tau^{-1}(w)=\frac{b-dw}{cw-a}\quad\text{and}\quad\tau^{\prime}(z)=\frac
{ad-bc}{(cz+d)^{2}},\quad w\in\Omega,z\in\mathbb{U}.
\]
Hence
\[
\tau^{\prime}(\tau^{-1}(w))=\frac{ad-bc}{(c\tau^{-1}(w)+d)^{2}}=\frac
{ad-bc}{\left[  c\left(  \frac{b-dw}{cw-a}\right)  +d\right]  ^{2}}%
=\frac{(cw-a)^{2}}{ad-bc}%
\]
and
\[
K_{\tau^{-1}(u)}(\tau^{-1}(w))=\frac{1}{1-\overline{\tau^{-1}(u)}\tau^{-1}%
(w)}=\frac{\overline{(cu-a)}(cw-a)}{\overline{(cu-a)}(cw-a)-\overline
{(b-du)}(b-dw)}.
\]
Combining these equalities, it turns out that $k_{u}(w)$ is given by
\begin{align*}
k_{u}(w) &  =\left(  \overline
{(\tau^{\prime}(\tau^{-1}(u)))}^{-1/2}V^{-1}K_{\tau^{-1}(u)}\right)  (w)\\
&=\overline{(\tau^{\prime}(\tau^{-1}(u)))}^{-1/2}(\tau^{\prime
}(\tau^{-1}(w)))^{-1/2}K_{\tau^{-1}(u)}(\tau^{-1}(w))\\
&  =\overline{\frac{(ad-bc)}{\overline{cu-a}}}^{1/2}\frac{(ad-bc)}{cw-a}%
^{1/2}\frac{\overline{(cu-a)}(cw-a)}{\overline{(cu-a)}(cw-a)-\overline
{(b-du)}(b-dw)}\\
&  =\frac{|ad-bc|}{|a|^{2}-|b|^{2}+(b\overline{d}-a\overline{c})\overline
{u}+[(|c|^{2}-|d|^{2})\overline{u}+\overline{b}d-\overline{a}c]w}%
\end{align*}
as desired.
\end{proof}

It follows easily from this lemma that if $\tau$ is the identity or $\tau
_{1},$ then
\begin{align*}
K_{u}(z)=\frac{1}{1-\overline{u}z} \quad\text{and}\quad k_{u}(w)=\frac
{1}{\overline{u}+w}, \quad z\in\mathbb{U}, w\in\mathbb{C}_{+},
\end{align*}
are the reproducing kernels on $H^{2}(\mathbb{U})$ and $H^{2}(\mathbb{C}%
_{+}),$ respectively.

\begin{cor}
For every $u\in \Omega,$ let $\delta_u$ be the linear functional $\delta_u:H^2(\Omega)\longrightarrow \mathbb{C}$ defined by $\delta_u(f)=f(u).$ Then $\delta_u$ is bounded with $\left\| \delta_u\right\| =1/(|\sqrt{ \tau^{\prime}(\tau^{-1}(u))}|\sqrt{1-|\tau^{-1}(u)|^2})$ and for each $f\in H^2(\Omega)$
\begin{align*}
|f(u)| \leq  \frac{\left\| f\right\|}{ |\sqrt{ \tau^{\prime}(\tau^{-1}(u))}|\sqrt{1-|\tau^{-1}(u)|^2}}, \quad u\in \Omega.
\end{align*}
\end{cor}
\begin{proof}
For $f\in H^2(\Omega) $ and $u\in \Omega,$ we have $|f(u)|=|\left\langle f, k_u\right\rangle |$. Using the Cauchy-Schwarz inequality and Lemma \ref{024} we get
\begin{align*}
|\delta_u(f)|=|\left\langle f, k_u\right\rangle |\leq \left\| f\right\|  \left\| k_u\right\|= \left\| f\right\|\frac{\parallel K_{\tau^{-1}(u)}\parallel_2}{| \sqrt{\tau^{\prime}(\tau^{-1}(u))}|}
\end{align*}
Thus, $\delta_u$ is a bounded linear functional of norm at most $\parallel 
K_{\tau^{-1}(u)}\parallel_2 /|\sqrt{\tau^{\prime}(\tau^{-1}(u))}|.$ Applying again Lemma \ref{024}, we obtain 
\begin{align*}
\delta_u\left( \frac{1}{\left\| k_u\right\|}k_u \right) =\left\| k_u\right\|
= \frac{\parallel K_{\tau^{-1}(u)}\parallel_2}{| \sqrt{\tau^{\prime}(\tau^{-1}(u))}|}.
\end{align*}
This completes the proof.
\end{proof}
The proof of next result is analogous of the version for the Hardy space $H^2(\mathbb{U}).$

\begin{cor}
\label{023} If $C_{\Phi}$ is a bounded operator on $H^2(\Omega)$ then $C_{\Phi}^{*}k_{u}=k_{\Phi(u)},$ for
each $u\in\Omega.$
\end{cor}

\begin{cor}\label{0105} If $C_{\Phi}$ is a bounded operator on $H^2(\Omega)$ then 
\begin{align}\label{0110}
\sup\left\lbrace \frac{|\sqrt{\tau'(\tau^{-1}(u))}|\sqrt{1-|\tau^{-1}(u)|^2}}{|\sqrt{\tau'(\tau^{-1}(\Phi(\tau(u))))}|\sqrt{1-|\tau^{-1}(\Phi(\tau(u)))|^2}}:u\in \Omega\right\rbrace  < \infty.
\end{align}
\end{cor}
\begin{proof} Since $C_{\Phi}$ is bounded, it follows from Corollary \ref{0105} that
\begin{align*}
\frac{\left\| k_{\Phi(u)}\right\|}{\left\| k_u\right\| }=\frac{\left\| C_{\Phi}^*k_u\right\|}{\left\| k_u\right\| } \leq \left\| C_{\Phi}\right\| , \quad u\in \Omega.
\end{align*} 
 Lemma \ref{024} gives
\begin{align*}
\frac{|\sqrt{\tau'(\tau^{-1}(u))}|\sqrt{1-|\tau^{-1}(u)|^2}}{|\sqrt{\tau'(\tau^{-1}(\Phi(\tau(u))))}|\sqrt{1-|\tau^{-1}(\Phi(\tau(u)))|^2}}\leq \left\| C_{\Phi}\right\| \quad u\in \Omega.
\end{align*}
\end{proof}
For analytic self-maps of $\mathbb{C}_+,$ we have that a composition operator is bounded on $H^2(\mathbb{C}_+)$ if and only if its inducing symbol satisfies the condition \eqref{0110} (see \cite[Theorem 3.1]{Elliot-Jury}).

\begin{rem}
\label{078} We frequently use the following fact: if $\tau(z)=(az+b)/(cz+d)$
is a Riemann map with $|c|^{2}=|d|^{2}$ then $b\overline{d}\neq a\overline
{c}.$ Indeed, since $\tau$ is a Riemann map, the equality $|c|^{2}=|d|^{2}$
implies that $c\neq0$ and $d\neq0.$ Hence
\begin{align*}
b\overline{d}-a\overline{c}=\frac{b}{d}|d|^{2}-\frac{a}{c}|c|^{2}%
=|d|^{2}\left(  \frac{b}{d}-\frac{a}{c}\right)  =|d|^{2}\left(  \frac
{bc-ad}{dc}\right)  \neq0.
\end{align*}
\end{rem}

Let $\Phi$ be an analytic self-map of $\Omega$ with a fixed point in $\Omega.$
Assuming $C_{\Phi}\neq I$ is cohyponormal on $H^{2}(\Omega),$ we use Lemma
\ref{024} to determine the fixed point of $\Phi.$

\begin{pps}
\label{081} Let $\tau$ be the linear fractional Riemann map $z\in \mathbb{U} \mapsto (az+b)/(cz+d)$  and let $\Phi$ be an analytic self-map of $\Omega.$ If
$\Phi(u)=u$ for some $u\in\Omega$ and $C_{\Phi}\neq I$ is cohyponormal on
$H^{2}(\Omega)$ then $|c|<|d|$ and
\[
u=(a\overline{c}-b\overline{d})/(|c|^{2}-|d|^{2}).
\]

\end{pps}

\begin{proof}
First we note that, by Corollary \ref{023}, $C_{\Phi}^{\ast}k_{u}=k_{u}.$
Since $C_{\Phi}$ is cohyponormal, we have $\mathrm{Ker}(C_{\Phi}^{*}%
-\overline{\lambda}I)\subset\mathrm{Ker}(C_{\Phi}-\lambda I),$ for each
$\lambda\in\mathbb{C}.$ Hence
$C_{\Phi}k_{u}=k_{u},$ and by Lemma \ref{024}, we obtain
\begin{gather*}
|a|^{2}-|b|^{2}+(b\overline{d}-a\overline{c})\overline{u}+[(|c|^{2}%
-|d|^{2})\overline{u}+\overline{b}d-\overline{a}c]\Phi(w)\\
=|a|^{2}-|b|^{2}+(b\overline{d}-a\overline{c})\overline{u}+[(|c|^{2}%
-|d|^{2})\overline{u}+\overline{b}d-\overline{a}c]w
\end{gather*}
or equivalently
\[
\lbrack(|c|^{2}-|d|^{2})\overline{u}+\overline{b}d-\overline{a}c](\Phi
(w)-w)=0.
\]
Since $C_{\Phi}\neq I,$ it follows that $\Phi$ is not the identity. Hence,
$(|c|^{2}-|d|^{2})\overline{u}+\overline{b}d-\overline{a}c=0.$ By Remark
\ref{078}, if $|c|^{2}=|d|^{2}$ then $\overline{b}d\neq\overline{a}c.$ This
implies that we must have $|c|^{2}\neq|d|^{2},$ and hence
\[
u=(a\overline{c}-b\overline{d})/(|c|^{2}-|d|^{2}).
\]
Since $u\in\tau(\mathbb{U})$ there is $z_{u}\in\mathbb{U}$ such that
$u=(az_{u}+b)/(cz_{u}+d).$ A simple computation gives $z_{u}\overline
{d}(ad-bc)=\overline{c}(bc-ad)$ and hence $z_{u}=-\overline{c}/\overline{d}.$
\end{proof}

\begin{cor}
If $\Omega=\mathbb{C}_{+}$ or $\Omega=\Pi^{+}$ then $H^{2}(\Omega)$ does not
contain normal bounded composition operators $C_{\Phi}$ with $\Phi$ fixing a
point of $\Omega.$
\end{cor}

In \cite{Gajath}, Gunatillake studied the bounded Hermitian composition
operators on the Hardy-Smirnov spaces $H^{2}(\Omega).$ He showed that if
$C_{\Phi}$ is a bounded Hermitian operator on $H^{2}(\Omega)$ for some
constant map $\Phi$, then $\Omega$ is a disc (see \cite[Lemma 3.1]{Gajath}).
Now we will show that \cite[Lemma 3.1]{Gajath} holds for cohyponormal
composition operators induced by constant maps.
First we show that if $H^{2}(\Omega)$ supports a cohyponormal composition
operator induced by a constant map then $\tau$ is a linear fractional Riemann
map, and then we apply the remark below to conclude that $\tau(\mathbb{U})$ is
a disc.

\begin{rem}
\label{080} For each $\alpha\in\mathbb{U},$ the set $K_{\alpha}(\mathbb{U})$
is a disc. Indeed, let us denote $\varphi_{\alpha}:\mathbb{U}\longrightarrow
\mathbb{U}$ the automorphism of $\mathbb{U}$ given by $\varphi_{\alpha
}(z)=(\alpha-z)/(1-\overline{\alpha}z).$ Then, it is enough to show that
$K_{\alpha}(\varphi_{\alpha}(\mathbb{U}))$ is a disc. For each $z\in
\mathbb{U},$ we have
\begin{align*}
K_{\alpha}(\varphi_{\alpha}(z))  &  =\frac{1}{1-\overline{\alpha}%
\varphi_{\alpha}(z)}\\
&  =\frac{1-\overline{\alpha}z}{1-\overline{\alpha}z-|\alpha|^{2}%
+\overline{\alpha}z}\\
&  = \left(  \frac{\overline{\alpha}}{|\alpha|^{2}-1}\right)  z+\frac
{1}{1-|\alpha|^{2}}.
\end{align*}
and hence $K_{\alpha}(\mathbb{U})$ is a disc.
\end{rem}

\begin{pps}
\label{029} Let $\Phi$ be a constant self-map of $\Omega.$ If $C_{\Phi}$ is a
cohyponormal operator on $H^{2}(\Omega),$ then $\Omega$ is a disc.
\end{pps}

\begin{proof}
Suppose that $\Phi\equiv w_{0}$ for some $w_{0}\in\Omega$ and that $C_{\Phi}$
is a cohyponormal operator on $H^{2}(\Omega).$
Since $W_{\varphi}$ and $C_{\Phi}$ are isometrically similar, it follows that
$W_{\varphi}$ also is cohyponormal. Now, we observe that a simple
computation gives $\varphi\equiv\alpha,$ where $\alpha=\tau^{-1}(w_{0}).$
Hence
\begin{align}\label{0104}
(W_{\varphi}f)(z)=\sqrt{\frac{\tau^{\prime}(z)}{\tau^{\prime}(\alpha)}%
}f(\alpha),\quad f\in H^{2}(\mathbb{U}),z\in\mathbb{U}.
\end{align}
Since $\alpha$ is a fixed point of $\varphi$ in $\mathbb{U},$ it follows from
Lemma \ref{05} that $K_{\alpha}$ is an eigenvector of $W_{\varphi}^*$ associated
to the eigenvalue $1.$ By the cohyponormality of $W_{\varphi},$ $K_{\alpha}$
is also an eigenvector of $W_{\varphi}$ associated to the eigenvalue $1,$
which combined with \eqref{0104} gives
\[
\sqrt{\frac{\tau^{\prime}(z)}{\tau^{\prime}(\alpha)}}=\frac{1-|\alpha|^{2}%
}{1-\overline{\alpha}z},\quad z\in\mathbb{U}.
\]
Hence
\begin{align}\label{079}
\tau^{\prime}(z)=\frac{c_{1}}{(1-\overline{\alpha}z)^{2}},\quad z\in\mathbb{U}%
\end{align}
where $c_{1}=(1-|\alpha|^{2})^{2}\tau^{\prime}(\alpha).$ Now, we study the
cases, $\alpha=0$ and $\alpha\neq0,$ separately.

\textbf{Case 1:} $\alpha=0.$\newline In this case, $\tau^{\prime}$ is the
constant function identically to $c_{1}.$ Since $\tau$ is a Riemann map,
$c_{1}\neq0.$ Therefore, $\tau(z)=c_{1}z+b$ for some $b\in\mathbb{C},$ which
implies that $\Omega$ is a disc.

\textbf{Case 2:} $\alpha\neq0.$\newline In this case, \eqref{079} implies that
the Riemann map $\tau$ must have the following form
\[
\tau(z)=\frac{c_{1}}{\overline{\alpha}}K_{\alpha}(z)+c_{2},\quad
z\in\mathbb{U}.
\]
As we saw in Remark \ref{080}, the $K_{\alpha}(\mathbb{U})$ is a disc. Hence
the range of $\tau$ must also be a disc.
\end{proof}

\section{Hermitian and unitary composition operators}

\label{097} Assume that $C_{\Phi}$ and $C_{\Psi}$ are Hermitian and unitary,
respectively. It is standard in composition operator theory to study the equalities
\[
C_{\Phi}k_{u}=C_{\Phi}^{\ast}k_{u}\quad\text{and}\quad C_{\Psi}C_{\Psi}^{\ast
}k_{u}=k_{u},\quad u\in\Omega
\]
to determine which symbols $\Phi$ and $\Psi$ induce Hermitian and unitary
composition operators, respectively (see, for instance, \cite{Paul2, Cowen-Ko,
Gajath, Jovic}). Here we shall introduce a different approach. We first
observe that if $C_{\Phi}$ is Hermitian and $C_{\Psi}$ is unitary, then
$C_{\Phi}^{\ast}$ and $C_{\Psi}^{\ast}$ are composition operators and thus,
when $\tau$ is a fractional linear Riemann
map, we use the expression of the reproducing kernels given by Lemma \ref{024}
to describe all the composition operators $C_{\Phi}$ whose adjoints are
multiple of other composition operators on $H^{2}(\Omega).$ So we use this
description to determine the symbols $\Phi$ that induce Hermitian or unitary
composition operators on $H^{2}(\Omega)$ (see Theorems \ref{059} and \ref{060}).





\begin{rem}
\label{083} By Lemma \ref{029}, it follows that $C_{\Phi}^{*}$ is never the
zero operator. Hence, if $\Phi$ and $\Phi_{\star}$ are analytic self-maps of
$\Omega$ with $C_{\Phi}^{*}=\lambda C_{\Phi_{\star}}$ for some complex number
$\lambda,$ we must have $\lambda\neq0.$
\end{rem}

Theorem below provides a formula for calculating the adjoints of certain
classes of composition operators on $H^{2}(\Omega).$

\begin{thm}
\label{038} Let $\tau$ be the linear fractional Riemann map $z\in \mathbb{U} \mapsto (az+b)/(cz+d),\lambda \in \mathbb{C}\backslash\{0\},$ $\Phi$ and $\Phi_{\star}$ be analytic self-maps of $\Omega$
such that $C_{\Phi}$ and $C_{\Phi_{\star}}$ are bounded on $H^{2}(\Omega).$ Then
$C_{\Phi}^{\ast}=\lambda C_{\Phi_{\star}}$ if and only if one of the following cases occurs

\begin{enumerate}
\item If $|c|^{2}=|d|^{2}$ then there is a complex number $r$ such that
\begin{align}\label{035}
\Phi(w)=\overline{\lambda}^{-1}w+r\quad\text{and}\quad\Phi_{\star}(w)=\lambda
w+\frac{(\lambda-1)(|a|^{2}-|b|^{2})+\lambda\overline{r}(b\overline
{d}-a\overline{c})}{\overline{b}d-\overline{a}c}.
\end{align}

\item If $|c|^{2}\neq|d|^{2}$ and $(|a|^{2}-|b|^{2})(|c|^{2}-|d|^{2}%
)=|b\overline{d}-a\overline{c}|^{2}$ then there is a complex number $r$ such
that%
\begin{align}\label{036}
\Phi(w)=\lambda w+\frac{(r-1)(b\overline{d}-a\overline{c})}{|c|^{2}-|d|^{2}}%
\quad\text{and}\quad\Phi_{\star}(w)=\lambda \overline{r}w+\frac{(\lambda
\overline{r}-1)(b\overline{d}-a\overline{c})}{|c|^{2}-|d|^{2}}.
\end{align}

\item If $|c|^{2}\neq|d|^{2}$ and $(|a|^{2}-|b|^{2})(|c|^{2}-|d|^{2}%
)\neq|b\overline{d}-a\overline{c}|^{2}$ then $\lambda=1$ and there is
$r\in\mathbb{C}$ such that%
\begin{align}\label{043}
\Phi(w)=rw+\frac{(r-1)(b\overline{d}-a\overline{c})}{|c|^{2}-|d|^{2}}%
\quad\text{and}\quad\Phi_{\star}(w)=\overline{r}w+\frac{(\overline
{r}-1)(b\overline{d}-a\overline{c})}{|c|^{2}-|d|^{2}}.
\end{align}

\end{enumerate}
\end{thm}

\begin{proof}
By Corollary \ref{023}, we have that $C_{\Phi}^{\ast}=\lambda C_{\Phi_{\star}%
}$ for some complex number $\lambda\neq0$ if and only if
\begin{align*}
k_{\Phi(u)}(w)=\lambda k_{u}(\Phi_{\star}(w)),\quad u,w\in\Omega
\end{align*}
or equivalently
\begin{align}\label{031}
&  \lambda\left\{  |a|^{2}-|b|^{2}+(b\overline{d}-a\overline{c})\overline
{\Phi(u)}+[(|c|^{2}-|d|^{2})\overline{\Phi(u)}+\overline{b}d-\overline
{a}c]w\right\} \\
&  =|a|^{2}-|b|^{2}+(b\overline{d}-a\overline{c})\overline{u}+[(|c|^{2}%
-|d|^{2})\overline{u}+\overline{b}d-\overline{a}c]\Phi_{\star}(w),\quad
u,w\in\Omega.\nonumber
\end{align}
Now we study the cases, $|c|^{2}=|d|^{2}$ and $|c|^{2}\neq|d|^{2},$ separately.

\textbf{Case 1:} $|c|^{2}=|d|^{2}.$ \newline In this case, it turns out that
\eqref{031} is
\begin{align}\label{032}
&  \overline{\lambda}[  |a|^{2}-|b|^{2}+(\overline{b}d-\overline
{a}c)\Phi(u)+(b\overline{d}-a\overline{c})\overline{w}] \\
&  =|a|^{2}-|b|^{2}+(\overline{b}d-\overline{a}c)u+(b\overline{d}%
-a\overline{c})\overline{\Phi_{\star}(w)},\quad u,w\in\Omega.\nonumber
\end{align}
Differentiating \eqref{032} with respect to $u,$ we obtain
\begin{align}\label{033}
\overline{\lambda}\Phi^{\prime}(u)(\overline{b}d-\overline{a}c)=(\overline
{b}d-\overline{a}c),\quad u\in\Omega.
\end{align}
By Remark \ref{078}, we have $\overline{b}d\neq\overline{a}c$ and hence
\eqref{033} gives $\Phi(u)=\overline{\lambda}^{-1}u+r$ for some complex number
$r.$ Taking complex conjugate in \eqref{032} and differentiating with respect
to $w,$ we obtain $\Phi_{\star}^{\prime}(w)=\lambda.$ Hence, $\Phi_{\star
}(w)=\lambda w+s$ for some complex number $s.$ Replacing the expressions of
$\Phi(u)$ and $\Phi_{\star}(w)$ in \eqref{032}, we have
\begin{align*}
&  \overline{\lambda}(|a|^{2}-|b|^{2})+(\overline{b}d-\overline{a}%
c)u+(\overline{b}d-\overline{a}c)\overline{\lambda}r+(b\overline{d}%
-a\overline{c})\overline{\lambda w}\\
&  =|a|^{2}-|b|^{2}+(\overline{b}d-\overline{a}c)u+(b\overline{d}%
-a\overline{c})\overline{\lambda w}+(b\overline{d}-a\overline{c})\overline
{s},\quad u,w\in\Omega,
\end{align*}
or equivalently,
\[
(\overline{b}d-\overline{a}c)s
=(\lambda-1)(|a|^{2}-|b|^{2})+\lambda\overline{r}(b\overline
{d}-a\overline{c}).
\]
Thus, we obtain \eqref{035}.

\textbf{Case 2:} $|c|^{2}\neq|d|^{2}.$\newline Differentiating \eqref{031}
with respect to $w,$ we obtain
\[
\lambda\lbrack(|c|^{2}-|d|^{2})\overline{\Phi(u)}+\overline{b}d-\overline
{a}c]=[(|c|^{2}-|d|^{2})\overline{u}+\overline{b}d-\overline{a}c]\Phi_{\star
}^{\prime}(w),\quad u,w\in\Omega.
\]
This implies that
\[
\Phi(u)=ru+\frac{(r-1)(b\overline{d}-a\overline{c})}{|c|^{2}-|d|^{2}},\quad
u\in\Omega,
\]
where $\overline{r}=\Phi_{\star}^{\prime}(w)/\lambda.$ Taking the complex
conjugate in \eqref{031} and differentiating with respect to $u,$ we get that
\[
\Phi_{\star}(w)=\lambda\overline{r}w+\frac{(\lambda\overline{r}-1)(b\overline
{d}-a\overline{c})}{|c|^{2}-|d|^{2}},\quad w\in\Omega.
\]
Conversely, if $|c|^2=|d|^2,$ $\Phi$ and $\Phi_{\star}$ are as in \eqref{035}, then a straightforward computation shows that \eqref{031} holds. If $|c|^2\neq |d|^2,$ $\Phi$ and $\Phi_{\star}$ are as in \eqref{036}, then the left side of \eqref{031} is
\begin{align}\label{0108}
\lambda\left\{  |a|^{2}-|b|^{2}+(|c|^{2}-|d|^{2})\overline{ru}w+\overline
{r}[(b\overline{d}-a\overline{c})\overline{u}+(\overline{b}d-\overline
{a}c)w]+\frac{(\overline{r}-1)|b\overline{d}-a\overline{c}|^{2}}%
{|c|^{2}-|d|^{2}}\right\}  .
\end{align}
 while the right
side of \eqref{031} is
\begin{align}\label{0109}
|a|^{2}-|b|^{2}+(|c|^{2}-|d|^{2})\lambda\overline{ru}w+\lambda\overline
{r}[(b\overline{d}-a\overline{c})\overline{u}+(\overline{b}d-\overline
{a}c)w]+\frac{(\lambda\overline{r}-1)|\overline{b}d-\overline{a}c|^{2}%
}{|c|^{2}-|d|^{2}}.
\end{align}
Putting these together, it follows that \eqref{0108} is equal to \eqref{0109} if and only if
\[
\lambda\left[  |a|^{2}-|b|^{2}+\frac{(\overline{r}-1)|b\overline{d}%
-a\overline{c}|^{2}}{|c|^{2}-|d|^{2}}\right]  =|a|^{2}-|b|^{2}+\frac
{(\lambda\overline{r}-1)|b\overline{d}-a\overline{c}|^{2}}{|c|^{2}-|d|^{2}},
\]
or equivalently,
\begin{align}\label{044}
(\lambda-1)\left(  |a|^{2}-|b|^{2}-\frac{|b\overline{d}-a\overline{c}|^{2}%
}{|c|^{2}-|d|^{2}}\right)  =0.
\end{align}
Therefore, \eqref{036} or \eqref{043} occurs if and only if \eqref{044} holds.
\end{proof}

Theorem \ref{038} recovers classical results of Hermitian and unitary
composition operators on $H^{2}(\mathbb{C}_{+})$ (see \cite[Theorems 2.4 and
3.1]{Matache Inv and normal}).

\begin{rem}
	The linear fractional self-maps $\Phi$ of $\Pi^{+}$ that induce bounded
	composition operators on $H^{2}(\Pi^{+})$ are of the following
	form
	\begin{align}
	\label{045}\Phi(w)=\lambda w+r,
	\end{align}
	where $\lambda\in(0,\infty)$ and $\mathrm{Im}(r)\geq0.$ In \cite{Eva},
	Gallardo-Guti\'{e}rrez and Montes-Rodr\'{\i}guez provide a formula for the
	adjoints of composition operators induced by such maps. The next result shows that  this formula can be obtained from Theorem \ref{038}.
\end{rem}

\begin{cor}
	If $\Phi$ is as in \eqref{045}, then $C_{\Phi}^{\ast}=\lambda^{-1}%
	C_{\Phi_{\star}},$ where $\Phi_{\star}(w)=\lambda^{-1}w-\lambda^{-1}%
	\overline{r}.$
\end{cor}

\begin{proof}
	Since $\tau_{2}(z)=i(1+z)/(1-z)$ is a linear fractional Riemann map from
	$\mathbb{U}$ onto $\Pi^{+}$ and $\Phi_{\star}(w)=\lambda^{-1}%
	w-\lambda^{-1}\overline{r}$ is a self-map of $\Pi^{+}$ which induces a
	bounded composition operator on $H^{2}(\Pi^{+}),$ Theorem \ref{038}
	gives $C_{\Phi}^{*}=\lambda^{-1}C_{\Phi_{\star}}.$
\end{proof}

\subsection{Hermitian composition operators}

Here, we assume that $\tau$ is a linear fractional Riemann map and we describe all symbols $\Phi$ whose
composition operators $C_{\Phi}$ are Hermitian on $H^{2}(\Omega).$ 

\begin{thm}
\label{059} Let $\tau$ be the linear fractional Riemann map $z\in \mathbb{U} \mapsto (az+b)/(cz+d)$ and $\Phi$ be an analytic self-map of $\Omega$ such that
$C_{\Phi}$ is a Hermitian composition operator on $H^{2}(\Omega)$.

\begin{enumerate}
\item \label{051} If $|c|^{2}=|d|^{2}$ then there is $r\in\mathbb{C}$ with
$\overline{r}(b\overline{d}-a\overline{c})\in\mathbb{R}$ such that
\begin{align}
\label{084}\Phi(w)=w+r.
\end{align}

\item \label{054} If $|c|^{2}\neq|d|^{2}$ then there is $r\in\mathbb{R}$ such
that
\begin{align}
\label{085}\Phi(w)=rw+\frac{(r-1)(b\overline{d}-a\overline{c})}{|c|^{2}%
-|d|^{2}}.
\end{align}

\end{enumerate}

Conversely, let $\Phi$ be as in \eqref{084} or \eqref{085} with conditions
\eqref{051} and \eqref{054}, respectively. If $\Phi$ is a self-map
of $\Omega$ such that $C_{\Phi}$ is bounded on $H^{2}(\Omega)$ then $C_{\Phi}$
is Hermitian.
\end{thm}

\begin{proof}
Since $C_{\Phi}$ is Hermitian, we have $C_{\Phi}^{*}=C_{\Phi}.$
Applying Theorem \ref{038} with $\lambda=1,$ we obtain $\Phi=\Phi_{\star}.$ To
find the expression of $\Phi,$ we must consider two cases:

\textbf{Case 1:} $|c|^{2}=|d|^{2}.$ \newline By Theorem \ref{038}, there is
$r\in\mathbb{C}$ such that
\begin{align*}
\Phi(w)=w+r \quad\text{and} \quad\Phi_{\star}(w)= w+\frac{\overline
{r}(b\overline{d}-a\overline{c})}{\overline{b}d-\overline{a}c}.
\end{align*}
Since $\Phi=\Phi_{\star},$ it follows that $r(\overline{b}d-\overline
{a}c)=\overline{r}(b\overline{d}-a\overline{c}).$

\textbf{Case 2:} $|c|^{2}\neq|d|^{2}.$\newline By Theorem \ref{038}, there is
$r\in\mathbb{C}$ such that
\[
\Phi(w)=rw+\frac{(r-1)(b\overline{d}-a\overline{c})}{|c|^{2}-|d|^{2}}%
\quad\text{and}\quad\Phi_{\star}(w)=\overline{r}w+\frac{(\overline
{r}-1)(b\overline{d}-a\overline{c})}{|c|^{2}-|d|^{2}}.
\]
Since $\Phi=\Phi_{\star},$ we have
\begin{align}
\label{053}rw+\frac{(r-1)(b\overline{d}-a\overline{c})}{|c|^{2}-|d|^{2}%
}=\overline{r}w+\frac{(\overline{r}-1)(b\overline{d}-a\overline{c})}%
{|c|^{2}-|d|^{2}},\quad w\in\Omega.
\end{align}
Differentiating \eqref{053}, we obtain $r=\overline{r}.$\newline Conversely,
first we assume that $\Phi$ is as in \eqref{084} with condition \eqref{051}.
Then $\Phi_{\star}(w)=w+\frac{\overline{r}(b\overline{d}-a\overline{c}%
)}{\overline{b}d-\overline{a}c}$ with $\overline{r}(b\overline{d}%
-a\overline{c})\in\mathbb{R}.$ Hence, we obtain
\[
\Phi_{\star}(w)=w+\frac{\overline{r}(b\overline{d}-a\overline{c})}%
{\overline{b}d-\overline{a}c}=w+\frac{r(\overline{b}d-\overline{a}%
c)}{\overline{b}d-\overline{a}c}=\Phi(w),\quad w\in\Omega.
\]
Since $C_{\Phi}$ is a bounded operator on 
$H^{2}(\Omega),$ it follows from Theorem \ref{038} that $C_{\Phi}^{\ast}=C_{\Phi_{\star
}},$ and hence $C_{\Phi}^{\ast}=C_{\Phi}.$ Now, we assume that $\Phi$ is as in
\eqref{085} with condition \eqref{054}. Then
\[
\Phi_{\star}(w)=\overline{r}w+\frac{(\overline{r}-1)(b\overline{d}%
-a\overline{c})}{|c|^{2}-|d|^{2}}%
\]
with $r\in\mathbb{R}$ (note that we can choose $\lambda=1$). So, it is
immediate that $\Phi=\Phi_{\star}.$ By applying again Theorem \ref{038}, we
obtain $C_{\Phi}^{\ast}=C_{\Phi_{\star}},$ and hence $C_{\Phi}^{\ast}=C_{\Phi
}.$
\end{proof}

\begin{exm}
It follows from Theorem \ref{059} that $C_{\Phi}$ is a bounded Hermitian composition operator on
$H^{2}(\mathbb{C}_{+})$ ({\it resp. $H^{2}(\mathbb{U})$}) if and only if $\Phi(w)=w+ir$ for some $r\geq0$ ({\it resp. $\Phi(w)=\lambda w$ for some
$\lambda\in[-1,1]$}).
\end{exm}



\subsection{Unitary composition operators}

Similarly to the previous section, we assume that $\tau$ is a 
linear fractional Riemann map and
we determine which symbols $\Phi$ induce unitary composition operators on $H^2(\Omega).$

\begin{thm}
\label{060}  Let $\tau$ be the linear fractional Riemann map $z\in \mathbb{U} \mapsto (az+b)/(cz+d)$ and $\Phi$ be an analytic self-map of $\Omega$ such that
$C_{\Phi}$ is a unitary bounded composition operator on $H^{2}(\Omega)$.

\begin{enumerate}
\item \label{055} If $|c|^{2}=|d|^{2}$ then there is $r\in\mathbb{C}$ with
$\mathrm{Re}(\overline{r}(b\overline{d}-a\overline{c}))=0$ such that
\begin{align}
\label{087}\Phi(w)=w+r.
\end{align}

\item \label{056} If $|c|^{2}\neq|d|^{2}$ then there is $r\in\mathbb{C}$ with
$|r|=1$ such that
\begin{align}
\label{088}\Phi(w)=rw+\frac{(r-1)(b\overline{d}-a\overline{c})}{|c|^{2}%
-|d|^{2}}%
\end{align}

\end{enumerate}

Conversely, let $\Phi$ be as in \eqref{087} or \eqref{088} with conditions
\eqref{055} and \eqref{056}, respectively. If $\Phi$ is a self-map
of $\Omega$ such that $C_{\Phi}$ invertible and bounded on $H^{2}(\Omega)$ then
$C_{\Phi}$ is unitary.
\end{thm}

\begin{proof}
By hypothesis $C_{\Phi}$ is unitary. Then $C_{\Phi}^{*}=C_{\Phi}^{-1}$ and
since $C_{\Phi}^{-1}=C_{\Phi^{-1}},$ we have $C_{\Phi}^{*}=C_{\Phi^{-1 }}.$
Applying Theorem \ref{038} with $\lambda=1,$ we obtain $\Phi^{-1}=\Phi_{\star
}.$ To find the expression of $\Phi,$ we consider two cases.

\textbf{Case 1:} $|c|^{2}=|d|^{2}.$ \newline By Theorem \ref{038}, there is
$r\in\mathbb{C}$ such that
\begin{align*}
\Phi(w)=w+r \quad\text{and} \quad\Phi^{-1}(w)= w+\frac{\overline{r}%
(b\overline{d}-a\overline{c})}{\overline{b}d-\overline{a}c}.
\end{align*}
Since $\Phi^{-1}$ also is given by $\Phi^{-1}(w)=w-r,$ we obtain
\begin{align*}
w-r= w+ \frac{\overline{r}(b\overline{d}-a\overline{c})}{\overline
{b}d-\overline{a}c}, \quad w\in\Omega.
\end{align*}
which implies $\mathrm{Re}(\overline{r}(b\overline{d}-a\overline{c}))=0.$

\textbf{Case 2:} $|c|^{2}\neq|d|^{2}.$\newline By Theorem \ref{038}, there is
$r\in\mathbb{C}$ such that
\[
\Phi(w)=rw+\frac{(r-1)(b\overline{d}-a\overline{c})}{|c|^{2}-|d|^{2}}%
\quad\text{and}\quad\Phi^{-1}=\overline{r}w+\frac{(\overline{r}-1)(b\overline
{d}-a\overline{c})}{|c|^{2}-|d|^{2}}.
\]
Since $\Phi^{-1}$ also is given by
\[
\Phi^{-1}(w)=\frac{w}{r}-\frac{(r-1)(b\overline{d}-a\overline{c})}%
{r(|c|^{2}-|d|^{2})},
\]
we get
\begin{align}
\label{042}\frac{w}{r}-\frac{(r-1)(b\overline{d}-a\overline{c})}%
{r(|c|^{2}-|d|^{2})}=\overline{r}w+\frac{(\overline{r}-1)(b\overline
{d}-a\overline{c})}{|c|^{2}-|d|^{2}},\quad w\in\Omega.
\end{align}
Differentiating \eqref{042}, it follows that $|r|=1.$ \newline Conversely,
first we assume that $\Phi$ is as in \eqref{087} with condition \eqref{055}.
Then $\Phi_{\star}(w)=w+\frac{\overline{r}(b\overline{d}-a\overline{c}%
)}{\overline{b}d-\overline{a}c}$ with $\mathrm{Re}(\overline{r}(b\overline
{d}-a\overline{c}))=0.$ Hence
\[
\Phi_{\star}(w)=w+\frac{\overline{r}(b\overline{d}-a\overline{c})}%
{\overline{b}d-\overline{a}c}=w+\frac{-r(\overline{b}d-\overline{a}%
c)}{\overline{b}d-\overline{a}c}=\Phi^{-1}(w).
\]
By Theorem \ref{038}, we obtain $C_{\Phi}^{\ast}=C_{\Phi_{\star}},$ and hence
$C_{\Phi}^{\ast}=C_{\Phi^{-1}}.$ Now, we assume that $\Phi$ is as in
\eqref{088} with condition \eqref{056}. Then
\[
\Phi_{\star}(w)=\overline{r}w+\frac{(\overline{r}-1)(b\overline{d}%
-a\overline{c})}{|c|^{2}-|d|^{2}}%
\]
with $|r|=1.$ Hence
\[
\Phi_{\star}(w)=\overline{r}w+\frac{(\overline{r}-1)(b\overline{d}%
-a\overline{c})}{|c|^{2}-|d|^{2}}=\Phi_{\star}(w)=\frac{|r|^{2}}{\overline{r}%
}w+\frac{(|r|^{2}-r)(b\overline{d}-a\overline{c})}{r(|c|^{2}-|d|^{2})}%
=\Phi^{-1}(w).
\]
Applying again Theorem \ref{038}, we have $C_{\Phi_{\star}}=C_{\Phi}^{-1}$
and hence $C_{\Phi}^{\ast}=C_{\Phi^{-1}}.$
\end{proof}

\begin{exm}
	It follows from Theorem \ref{060} that $C_{\Phi}$ is an unitary composition operator on
	$H^{2}(\mathbb{C}_{+})$ ({\it resp. $H^{2}(\mathbb{U})$}) if and only if $\Phi(w)=w+ir$ for some $r\in \mathbb{R}$ ({\it resp. $\Phi(w)=\lambda w$ for some
		$|\lambda|=1$}).
\end{exm}

\subsection{Normal composition operators}

We finish this section characterizing the symbols in Theorem \ref{038} that induce normal composition operators.

\begin{pps}
 Let $\tau$ be the linear fractional Riemann map $z\in \mathbb{U} \mapsto (az+b)/(cz+d)$ and $\lambda,r\in\mathbb{C}$ with $\lambda\neq0.$
 \begin{enumerate}
 \item If $|c|^2=|d|^2,$ $\Phi(w)=\overline{\lambda}^{-1}w+r$  and $\Phi_{\star}(w)=\lambda
 w+\frac{(\lambda-1)(|a|^{2}-|b|^{2})+\lambda\overline{r}(b\overline
 	{d}-a\overline{c})}{\overline{b}d-\overline{a}c}$
 are self-maps of $\Omega$ and $C_{\Phi}$ is bounded on $H^{2}(\Omega).$ Then $C_{\Phi}$ is normal if and only if
 \[
 (\lambda-1)(1-\overline{\lambda})(|a|^{2}-|b|^{2})=2\mathrm{Re}(\lambda
 \overline{r}(b\overline{d}-a\overline{c})(\overline{\lambda}-1)).
 \]
 \item If $|c|^2\neq |d|^2,$
 $\Phi(w)=rw+\frac{(r-1)(b\overline{d}-a\overline{c})}{|c|^{2}-|d|^{2}}%
 \quad\text{and}\quad\Phi_{\star}(w)=\overline{r}w+\frac{(\overline
 	{r}-1)(b\overline{d}-a\overline{c})}{|c|^{2}-|d|^{2}}$
 are self-maps of $\Omega$ and $C_{\Phi}$ is bounded on $H^{2}(\Omega).$ Then $C_{\Phi}$ is normal.
\end{enumerate}	
\end{pps}
\begin{proof}
By Theorem \ref{038}, in both cases $C_{\Phi}$ is normal if and only if $C_{\Phi}C_{\Phi_{\star}}=C_{\Phi_{\star}}C_{\Phi}.$ By Lemma \ref{024}, this is
equivalent to the $\Phi\circ\Phi_{\star}=\Phi_{\star}\circ\Phi.$ Now we study the cases, $|c|^2=|d|^2$ and $|c|^2\neq |d|^2,$ separately.
	
\textbf{Case 1:} $|c|^2=|d|^2.$ A straightforward
computation gives 
\[
(\Phi\circ\Phi_{\star})(w)=\frac{\lambda(\overline{b}d-\overline
	{a}c)w+(\lambda-1)(|a|^{2}-|b|^{2})+\lambda\overline{r}(b\overline
	{d}-a\overline{c})+\overline{\lambda}r(\overline{b}d-\overline{a}c)}%
{\overline{\lambda}(\overline{b}d-\overline{a}c)}%
\]
and 
\[
(\Phi_{\star}\circ\Phi)(w)=\frac{\lambda(\overline{b}d-\overline
	{a}c)w+|\lambda|^{2}r(\overline{b}d-\overline{a}c)+(|\lambda|^{2}%
	-\overline{\lambda})(|a|^{2}-|b|^{2})+|\lambda|^{2}\overline{r}(b\overline
	{d}-a\overline{c})}{\overline{\lambda}(\overline{b}d-\overline{a}c)}.
\]
Therefore, $\Phi\circ\Phi_{\star}=\Phi_{\star}\circ\Phi$ if and only if
\[
(\lambda-1)(|a|^{2}-|b|^{2})+2\mathrm{Re}(\lambda\overline{r}(b\overline
{d}-a\overline{c}))=(|\lambda|^{2}-\overline{\lambda})(|a|^{2}-|b|^{2}%
)+2\mathrm{Re}(|\lambda|^{2}\overline{r}(b\overline{d}-a\overline{c}))
\]
or equivalently
\[
(\lambda-1)(1-\overline{\lambda})(|a|^{2}-|b|^{2})=2\mathrm{Re}(\lambda
\overline{r}(b\overline{d}-a\overline{c})(\overline{\lambda}-1)).
\]	
\textbf{Case 2:} $|c|^2\neq |d|^2.$ Then	
\begin{align*}
(\Phi\circ\Phi_{\star})(w)=\lambda|r|^{2}(|c|^{2}-|d|^{2})w+\frac
{(b\overline{d}-a\overline{c})(|r|^{2}-1)}{|c|^{2}-|d|^{2}} = (\Phi_{\star}\circ\Phi)(w), \quad w\in \Omega
\end{align*}
which imples that $C_{\Phi}$ is normal.
\end{proof}

\section{Complex symmetric composition operators}

\label{062}

In Section \ref{097}, we describe all composition operators that are Hermitian
and unitary and, as mentioned in Introduction, such operators are always
complex symmetric. So for the purpose of classifying composition operators
that are complex symmetric, Theorems \ref{059} and \ref{060} provide partial
answers. However, as noted in \cite{Jung}, it is advantageous to determine
with respect to which conjugation an operator is complex symmetric. In some
situations, this allows us to obtain additional informations about the
operator, such as its spectrum and reduced spaces.

In \cite{Garc1}, Garcia and Hammond described all weighted composition
operators $W_{\psi, \varphi}$ that are $J$-symmetric on $H^{2}(\mathbb{U}),$
where $J:H^{2}(\mathbb{U})\longrightarrow H^{2}(\mathbb{U})$ is the
conjugation defined by
\begin{align}
\label{011}(Jf)(z)=\overline{f(\overline{z})}, \quad f\in H^{2}(\mathbb{U}),
z\in\mathbb{U}.
\end{align}
Using the same rule as $J,$ but in other spaces, \cite{Hai, Hai2, Han, Jung}
also described weighted composition operators and composition operators that
are $J$-symmetric.

In this section, we use the conjugation $J$ defined in \eqref{011} and the
isometry between $H^{2}(\mathbb{U})$ and $H^{2}(\Omega)$ to construct concrete
examples of conjugations on $H^{2}(\Omega).$ From these examples and assuming
that $\tau$ is a linear fractional Riemann map
with real coefficients, we explicit conjugations with respect to which
Hermitian and unitary composition operators are complex symmetric (see
Propositions \ref{063} and \ref{091}).

We begin this section by studying the simplest case of composition operators,
those induced by constant self-maps.

\begin{pps}
Let $\Phi$ be a constant self-map of $\Omega.$ If $C_{\Phi}$ is bounded on
$H^{2}(\Omega)$ then $C_{\Phi}$ is complex symmetric.
\end{pps}

\begin{proof}
As shown in the proof of Theorem \ref{029}, if $\Phi$ is constant then
$C_{\Phi}$ is similar to the operator $W_{\varphi}$ of rank one (the range of
$W_{\varphi}$ is generated by $(\tau^{\prime})^{1/2}$). By \cite[Theorem
2]{Wogen2}, operators of rank one are complex symmetric. Hence $C_{\Phi}$ is
complex symmetric.
\end{proof}

\subsection{$J_{\Omega}$-symmetry}


By Lemma \ref{065} the conjugation $J$ allows us to induce a conjugation on
the Hardy space $H^{2}(\Omega)$ defining
\begin{align}\label{099}
J_{\Omega}=V^{-1}JV,
\end{align}
where $V$ is the isometric isomorphism of Proposition \ref{03}. In the next
result we determine the expression for $J_{\Omega},$ and we will see that if
$\tau$ is a linear fractional Riemann map wtih real coefficients then
$J_{\Omega}$ has the same rule that $J.$

\begin{pps}
\label{012} If $J_{\Omega}$ is the conjugation defined in \eqref{099}, then 
\begin{align*}
(J_{\Omega}f)(w)= (\tau^{\prime}(\tau^{-1}(w)))^{-1/2}( \overline{\tau
^{\prime}( \overline{\tau^{-1}(w)}}))^{1/2} \overline{f( \tau( \overline
{\tau^{-1}(w)}))},\quad f\in H^{2}(\Omega),w\in\Omega.
\end{align*}
In particular, if $\tau$ is the linear fractional Riemann map $\tau
(z)=(az+b)/(cz+d)$ then
\begin{align*}
(J_{\Omega}f)(w)=\overline{f\left(  \frac{a\overline{b}-b\overline
{a}-a\overline{dw}+b\overline{cw}}{c\overline{b}-d\overline{a}-c\overline
{dw}+d\overline{cw}} \right)  }, \quad f\in H^{2}(\Omega), w\in\Omega
\end{align*}
and, if $a,b,c$ and $d$ are real numbers, then
\begin{align*}
(J_{\Omega}f)(w)=\overline{f(\overline{w})}, \quad f\in H^{2}(\Omega),
w\in\Omega.
\end{align*}

\end{pps}

\begin{proof}
For each $f\in H^{2}(\Omega)$ and $w\in\Omega,$ we have
\begin{align}
(J_{\Omega}f)(w)  &  =(\tau^{\prime}(\tau^{-1}(w)))^{-1/2}(JVf)(\tau
^{-1}(w))\nonumber\label{0100}\\
&  =(\tau^{\prime}(\tau^{-1}(w)))^{-1/2}\overline{(Vf)(\overline{\tau^{-1}%
(w)})}\nonumber\\
&  =(\tau^{\prime}(\tau^{-1}(w)))^{-1/2}(\overline{\tau^{\prime}%
(\overline{\tau^{-1}(w)}}))^{1/2}\overline{f(\tau(\overline{\tau^{-1}(w)}))}.
\end{align}
Assuming that $\tau$ is the linear fractional Riemann map $\tau
(z)=(az+b)(cz+d),$ then a simple computation gives $\overline{\tau^{\prime
}(\overline{\tau^{-1}(w)})}=\tau^{\prime}(\tau^{-1}(w))$ for each $w\in\Omega.$ In
view of \eqref{0100}, we need to compute $\tau(\overline{\tau^{-1}(w)}).$ Note
that
\[
\overline{\tau^{-1}(w)}=\frac{\overline{b-dw}}{\overline{cw-a}}%
\]
and hence
\[
\tau(\overline{\tau^{-1}(w)})=\frac{a\left(  \frac{\overline{b-dw}}%
{\overline{cw-a}}\right)  +b}{c\left(  \frac{\overline{b-dw}}{\overline{cw-a}%
}\right)  +d}=\frac{a\overline{b}-b\overline{a}-a\overline{dw}+b\overline{cw}%
}{c\overline{b}-c\overline{dw}+d\overline{cw}-d\overline{a}}%
\]
which implies
\begin{align}
(\tau^{\prime}(\tau^{-1}(w)))^{-1/2}(\overline{\tau^{\prime}(\overline
{\tau^{-1}(w)}}))^{1/2}\overline{f(\tau(\overline{\tau^{-1}(w)}))}%
=\overline{f\left(  \frac{a\overline{b}-b\overline{a}-a\overline
{dw}+b\overline{cw}}{c\overline{b}-d\overline{a}-c\overline{dw}+d\overline
{cw}}\right)  }.
\end{align}
If $a,b,c,$ and $d$ are real numbers, then
\[
\frac{a\overline{b}-b\overline{a}-a\overline{dw}+b\overline{cw}}{c\overline
{b}-d\overline{a}-c\overline{dw}+d\overline{cw}}=\frac{-ad\overline
{w}+bc\overline{w}}{-da+cb}=\frac{(-ad+bc)\overline{w}}{-ad+bc}=\overline{w}%
\]
which gives $(J_{\Omega}f)(w)=\overline{f(\overline{w})}.$
\end{proof}




Assuming that $\tau$ is a linear fractional Riemann map with real
coefficients, we describe all composition operators $C_{\Phi}$ that
are $J_{\Omega}$-symmetric on $H^{2}(\Omega).$

\begin{pps}
\label{063} Let $\tau$ be the linear fractional Riemann map $z\in \mathbb{U} \mapsto (az+b)/(cz+d)$ with real
coefficients  and let $\Phi$ be an analytic self-map of $\Omega$ such
that $C_{\Phi}$ is $J_{\Omega}$-symmetric on $H^{2}(\Omega) $.

\begin{enumerate}
\item \label{048} If $|c|^{2}=|d|^{2}$ then there is $r\in\mathbb{C}$ such
that
\begin{align}
\label{070}\Phi(w)=w+r.
\end{align}

\item \label{064} If $|c|^{2}\neq|d|^{2}$ then there is $r \in\mathbb{C}$ such
that
\begin{align}
\label{071}\Phi(w)=r w+\frac{(r-1)(bd-ac)}{|c|^{2}-|d|^{2}}.
\end{align}

\end{enumerate}

Conversely, let $\Phi$ be as in \eqref{070} or \eqref{071} with conditions
\eqref{048} and \eqref{064}, respectively. If $\Phi$ is a self-map
of $\Omega$ such that $C_{\Phi}$ is bounded on $H^{2}(\Omega)$ then $C_{\Phi}$
is $J_{\Omega}$-symmetric.
\end{pps}

\begin{proof}
Suppose that $C_{\Phi}$ is $J_{\Omega}$-symmetric. Then $J_{\Omega}C_{\Phi
}^{*}k_{u}=C_{\Phi}J_{\Omega}k_{u}$ for all $u\in\Omega.$ By Lemma \ref{024},
we get
\begin{gather}
\label{0106}(bd-ac)\Phi(u)+[(|c|^{2}-|d|^{2})\Phi(u)+bd-ac]w\\
=(bd-ac)u+[(|c|^{2}-|d|^{2})u+bd-ac]\Phi(w), \quad u,w\in\Omega.\nonumber
\end{gather}
Differentiating \eqref{0106} with respect to $u,$ we obtain
\begin{gather}
\label{0107}(bd-ac)\Phi^{\prime}(u)+(|c|^2-|d|^{2})\Phi^{\prime}(u)w\\
=(bd-ac)+(|c|^{2}-|d|^{2})\Phi(w), \quad u, w\in\Omega.\nonumber
\end{gather}
If $|c|^{2}=|d|^{2},$ then \eqref{0107} gives $\Phi^{\prime
}(u)=1$ for each $u\in\Omega.$ Hence, $\Phi(u)=u+r$ for some $r\in\mathbb{C}.$
In the case, $|c|^{2}\neq|d|^{2},$ we have from \eqref{0107} that
\begin{align*}
\Phi(w)=\Phi^{\prime}(u)w+\frac{(\Phi^{\prime}(u)-1)(bd-ac)}{|c|^{2}-|d|^{2}}.
\end{align*}
Conversely, in both cases we have $\overline{\Phi(\overline{w})}\in\Omega,$
for each $w\in\Omega.$ For $\Phi$ as in \eqref{070} of condition (\ref{048}),
we have $\Phi_{\star}(w)=\overline{\Phi(\overline{w})}=w+\overline{r}.$ By
Theorem \ref{038}, it follows that $C_{\Phi}^{*}=C_{\Phi_{\star}}.$ Hence, for
each $f\in H^{2}(\Omega)$ and $w\in\Omega,$ we obtain
\begin{align*}
(J_{\Omega}C_{\Phi}^{*}f)(w)=\overline{(C_{\Phi}^{*}f)(\overline{w}
)}=\overline{f(\Phi_{\star}(\overline{w}))}=\overline{f(\overline{w+r}
)}=(C_{\Phi}J_{\Omega}f)(w)
\end{align*}
which shows that $C_{\Phi}$ is $J_{\Omega}$-symmetric. Now, assume that $\Phi$
is as in \eqref{071} of condition \eqref{064}. Thus
\begin{align*}
\Phi_{\star}(w)=\overline{\Phi(\overline{w})}=\overline{r} w+\frac
{(\overline{r}-1)(bd-ac)}{|c|^{2}-|d^{2}}.
\end{align*}
By Theorem \ref{038}, it follows that $C_{\Phi}^{*}=C_{\Phi_{\star}}.$ Hence,
for each $f\in H^{2}(\Omega)$ and $w\in\Omega,$ we obtain
\begin{align*}
(J_{\Omega}C_{\Phi}^{*}f)(w)=\overline{C_{\Phi_{\star}}f(\overline{w}
)}=\overline{f\left(  \overline{rw}+\frac{(\overline{r}-1)(bd-ac)}
{|c|^{2}-|d|^{2}}\right)  }=(J_{\Omega}C_{\Phi}f)(w)
\end{align*}
and we conclude that $C_{\Phi}$ is $J_{\Omega}$-symmetric.
\end{proof}

From the proposition above, we see that if $\tau$ is a linear fractional map with real coefficients then the class of
$J_{\Omega}$-symmetric operators contains the bounded Hermitian and unitary
composition operators on $H^{2}(\Omega).$

Now, we observe that if $\Psi$ induces a unitary composition operator on
$H^{2}(\Omega),$ then Lemma \ref{065} ensures that $J_{\Omega,\Psi}:=C_{\Psi
}J_{\Omega}C_{\Psi}^{-1}$ is a conjugation on $H^{2}(\Omega).$

\begin{lem}
\label{098}  Let $\tau$ be the linear fractional Riemann map $z\in \mathbb{U} \mapsto (az+b)/(cz+d)$ with real coefficients and suppose that $\Psi$ induces an unitary composition
operator on $H^{2}(\Omega).$

\begin{enumerate}
\item \label{066} If $|c|^{2}=|d|^{2}$ then $\Psi(w)=w+si$ for some
$s\in\mathbb{R}$ and
\begin{align*}
(J_{\Omega, \Psi}f)(w)=\overline{f(\overline{w}-2si)}, \quad f\in H^{2}%
(\Omega),w\in\Omega.
\end{align*}

\item \label{067} If $|c|^{2}\neq|d|^{2}$ then $\Psi(w)=r w+\frac
{(r-1)(bd-ac)}{|c|^{2}-|d|^{2}}$ with $|r|=1$ and
\begin{align*}
(J_{\Omega, \Psi} f)(w)=\overline{f\left(  \frac{\overline{r}}{r}\overline
{w}-\frac{2\mathrm{Im}(r)(bd-ac)}{r(|c|^{2}-|d|^{2})}i\right)  }, \quad f\in
H^{2}(\Omega),w\in\Omega.
\end{align*}

\end{enumerate}
\end{lem}

\begin{proof}
For the case \eqref{066}, it follows from Theorem \ref{038} that $\Psi(w)=w+r$ for
some $r\in\mathbb{C}$ where $\mathrm{Re}(\overline{r}(b\overline{d}%
-a\overline{c}))=0.$ Since $|c|^{2}=|d|^{2},$ we have $b\overline
{d}\neq a\overline{c}$ and hence $\mathrm{Re}(r)=0.$ Considering $s=\mathrm{Im}(r),$
then for each $f\in H^{2}(\Omega)$ and $w\in\Omega,$ we obtain
\[
(C_{\Psi}J_{\Omega}C_{\Psi}^{-1}f)(w)=\overline{(C_{\Psi^{-1}}f)(\overline
{w}-is)}=\overline{f(\overline{w}-is-is)}=\overline{f(\overline{w}-2is)}.
\]
For the case \eqref{067}, Theorem \ref{038} provides
\[
\Psi(w)=rw+\frac{(r-1)(bd-ac)}{|c|^{2}-|d|^{2}}%
\]
with $|r|=1.$ Then for each $f\in H^{2}(\Omega)$ and $w\in\Omega,$ we have
\begin{align*}
(C_{\Psi}J_{\Omega}C_{\Psi}^{-1}f)(w)  &  =\overline{(C_{\Psi}^{-1}f)\left(
\overline{rw}+\frac{(\overline{r}-1)(bd-ac)}{|c|^{2}-|d|^{2}}\right)  }\\
&  =\overline{f\left(  \frac{\overline{r}}{r}\overline{w}+\frac{(\overline
{r}-1)(bd-ac)}{r(|c|^{2}-|d|^{2})}-\frac{(r-1)(bd-ac)}{r(|c|^{2}-|d|^{2}%
)}\right)  }\\
&  =\overline{f\left(  \frac{\overline{r}}{r}\overline{w}-\frac{2i\mathrm{Im}%
(r)(bd-ac)}{r(|c|^{2}-|d|^{2})}\right)  }%
\end{align*}
and this completes the proof.
\end{proof}

\begin{pps}
\label{091}  Let $\tau$ be the linear fractional Riemann map $z\in \mathbb{U} \mapsto (az+b)/(cz+d)$ with real coefficients and let $\Psi$ induce a unitary
composition operator and let $\Phi$ be an analytic self-map of $\Omega$ such that
$C_{\Phi}$ is $J_{\Omega,\Psi}$-symmetric on $H^{2}(\Omega).$ 

\begin{enumerate}
\item \label{068} If $|c|^{2}=|d|^{2}$ then there is $r\in\mathbb{C}$ such
that
\begin{align}
\label{093}\Phi(w)=w+r.
\end{align}

\item \label{069} If $|c|^{2}\neq|d|^{2}$ then there is $r\in\mathbb{C}$ such
that
\begin{align}
\label{094}\Phi(w)=r w+\frac{(r-1)(bd-ac)}{|c|^{2}-|d|^{2}}.
\end{align}

\end{enumerate}

Conversely, let $\Phi$ be as in \eqref{093} or \eqref{094} with conditions
\eqref{068} and \eqref{069}, respectively. If $\Phi$ is an analytic self-map
of $\Omega$ such that $C_{\Phi}$ is bounded on $H^{2}(\Omega),$ then $C_{\Phi
}$ is $J_{\Omega,\Psi}$-symmetric.
\end{pps}

\begin{proof}
Observe that in both cases $C_{\Phi}$ is $J_{\Omega,\Psi}$-symmetric if and
only if $J_{\Omega}C_{\Psi\circ\Phi\circ\Psi^{-1}}J_{\Omega}=C_{\Psi\circ
\Phi\circ\Psi^{-1}}^{*}.$ By Proposition \ref{063}, this is equivalent to
prove that $\Psi\circ\Phi\circ\Psi^{-1}$ is as in \eqref{070} and \eqref{071}, respectively.

\textbf{Case 1:} $|c|^{2}=|d|^{2}.$ \newline We have $\Psi(w)=w+si$ for some
$s\in\mathbb{R}.$ If $\Psi\circ\Phi\circ\Psi^{-1}$ is as in \eqref{070}, then
$(\Psi\circ\Phi\circ\Psi^{-1})(w)=w+r$ for some $r\in\mathbb{C}.$ Hence,
$\Phi(w-is)=(w-is)+r.$ Since $w\in\Omega\mapsto w+is$ is invertible, it
follows that $\Phi(w)=w+r.$ Conversely, if $\Phi(w)=w+r$ then a simple
computation shows that $(\Psi\circ\Phi\circ\Psi^{-1})(w)=w+r.$

\textbf{Case 2:} $|c|^{2}\neq|d|^{2}.$ \newline We have $\Psi(w)=sw+\frac
{(s-1)(bd-ac)}{|c|^{2}-|d|^{2}}$ with $|s|=1.$ If $\Psi\circ\Phi\circ\Psi
^{-1}$ is as in \eqref{071}, then
\[
(\Psi\circ\Phi\circ\Psi^{-1})(w)=rw+\frac{(r-1)(bd-ac)}{|c|^{2}-|d|^{2}}%
\]
for some $r\in\mathbb{C}.$ Hence
\begin{align*}
\Phi\left(  \frac{w}{s}-\frac{(s-1)(bd-ac)}{s(|c|^{2}-|d|^{2})}\right)   &
=\frac{rw}{s}+\frac{(r-1)(bd-ac)}{s(|c|^{2}-|d|^{2})}-\frac{(s-1)(bd-ac)}%
{s(|c|^{2}-|d|^{2})}\\
&  =r\left[  \frac{w}{s}-\frac{(s-1)(bd-ac)}{s(|c|^{2}-|d|^{2})}\right]
+\frac{(r-1)(bd-ac)}{|c|^{2}-|d|^{2}}.
\end{align*}
Since
\[
w\in\Omega\mapsto\frac{w}{s}-\frac{(s-1)(bd-ac)}{s(|c|^{2}-|d|^{2})}%
\]
is invertible, it follows that
\[
\Phi(w)=rw+\frac{(r-1)(bd-ac)}{|c|^{2}-|d|^{2}}.
\]
Conversely, if $\Phi(w)=rw+\frac{(r-1)(bd-ac)}{|c|^{2}-|d|^{2}},$ then for
each $w\in\Omega$ we have
\begin{align*}
(\Psi\circ\Phi\circ\Psi^{-1})(w)  &  =s\Phi(\Psi^{-1}(w))+\frac{(s-1)(bd-ac)}%
{|c|^{2}-|d|^{2}}\\
&  =s\left[  r\Psi^{-1}(w)+\frac{(r-1)(bd-ac)}{|c|^{2}-|d|^{2}}\right]
+\frac{(s-1)(bd-ac)}{|c|^{2}-|d|^{2}}\\
&  =s\left[  \frac{rw}{s}+\frac{(r-s)(bd-ac)}{s(|c|^{2}-|d|^{2})}\right]
+\frac{(s-1)(bd-ac)}{|c|^{2}-|d|^{2}}\\
&  =rw+\frac{(r-1)(bd-ac)}{|c|^{2}-|d|^{2}}
\end{align*}
and the proof is complete.
\end{proof}

\subsection{Unbounded domains}


We finish this section showing that the existence of fixed point of $\Phi$ together with geometric properties of
$\Omega$ gives information about complex symmetry of $C_{\Phi}$ on $H^{2}(\Omega).$
First we recall some concepts, which deals with the dynamics of analytic self-maps of $\mathbb{U}.$

	For each $n\in\mathbb{N},$ let $\varphi^{[n]}$ denote the $n$-th iterate of the analytic self-map $\varphi:\mathbb{U}\longrightarrow \mathbb{U}.$ If $\omega\in \overline{\mathbb{U}}$ is a point such that the sequence $\varphi^{[n]}$ converges uniformly on compact subsets of $\mathbb{U}$ to $\omega,$ then $\omega$ is said to be an \textit{attractive point} for $\varphi.$ 
		Denjoy-Wolff Theorem is concerned with the existence of attractive points for analytic self-maps of $\mathbb{U}.$ It states that (see \cite[Theorem 2.51]{Cowen-MacClauer}), if $\varphi$ is an analytic self-map of $\mathbb{U}$ which is not the identity or an elliptic automorphism of $\mathbb{U}$ (that is, an automorphism with a fixed point in $\mathbb{U}$), then there is a unique point $\omega\in\overline{\mathbb{U}}$ such that $\varphi^{[n]}(z)\rightarrow \omega$ when $n \rightarrow\infty,$ for each $z\in \mathbb{U}$ .	

\begin{thm}
\label{020}If $\Omega$ is unbounded and $\Phi$ is a non-automorphic analytic
self-map of $\Omega$ with a fixed point in $\Omega$ and
such that $C_{\Phi}$ is bounded on $H^{2}(\Omega)$, then $C_{\Phi}$ is not
complex symmetric.
\end{thm}

\begin{proof}
Let $\alpha\in \Omega$ be a fixed point of $\Phi.$ Suppose that $C_{\Phi}$ is complex symmetric. Then
$W_{\varphi}$ is also complex symmetric. Let $C$ be a conjugation on
$H^{2}(\mathbb{U})$ such $W_{\varphi}C=CW_{\varphi}^{*}$ and let $f=CK_{\beta},$
where $\beta=\tau^{-1}(\alpha).$ Since $\Phi(\alpha)=\alpha,$ we have
$W_{\varphi}^{*}K_{\beta}=K_{\beta}$ (see Lemma \ref{05}). Hence $W_{\varphi
}CK_{\beta}=CK_{\beta},$ that is, $f$ is an eigenvector of $W_{\varphi}$
corresponding to the eigenvalue $1.$ Thus, $g=f/(\tau^{\prime})^{1/2}$ is an
analytic function on $\mathbb{U}$ with
\begin{align*}
g\circ\varphi=\frac{f\circ\varphi}{(\tau^{\prime})^{1/2}} =
\frac{1}{(\tau^{\prime})^{1/2}}\left(  \frac{\tau^{\prime}}{\tau^{\prime}%
\circ\varphi}\right)  ^{1/2} f\circ\varphi= \frac{f}{(\tau^{\prime})^{1/2}} =
g.
\end{align*}
This 
implies that $g\circ \varphi^{[n]}=g$ for each $n\in \mathbb{N}.$ Since
$\varphi$ is not an automorphism of $\mathbb{U}$ and $\varphi(\beta)=\beta,$
it follows from Denjoy-Wolff Theorem that $\varphi^{[n]}$ converges to $\beta$
locally uniformly in $\mathbb{U}$ when $n \rightarrow\infty.$ Thus, $g=g(\beta).$ Since $f$ is not identically
zero, $g=f/(\tau^{\prime})^{1/2}$ is a non-zero constant. Hence $(\tau
^{\prime})^{1/2}=g^{-1}f\in H^{2}(\mathbb{U}).$ Thus, $\tau^{\prime}\in
H^{1}(\mathbb{U}),$ which implies that $\tau$ is continuous on $\overline
{\mathbb{U}}$ (see \cite[Theorem 3]{Duren}) and therefore $\tau(\mathbb{U}%
)=\Omega$ is bounded. This contradiction implies that $C_{\Phi}$ is not
complex symmetric.
\end{proof}

\begin{cor}
Let $\Omega$ be unbounded. If $\Phi$ is an analytic non-automorphic self-map
of $\Omega$ with a fixed point in $\Omega$ then
$C_{\Phi}$ is non-normal on $H^{2}(\Omega).$
\end{cor}

Considering that $\Omega$ is a bounded set, it is an open problem if Theorem \ref{020} holds for some $\Omega.$ It is worth mentioning that for the bounded case $\Omega=\mathbb{U},$
Bourdon and Noor showed that if $C_{\Phi}$ is complex symmetric on $H^2(\mathbb{U})$, then $\Phi$ has a fixed point in $\mathbb{U}$ (see \cite[Proposition, p. 106]{Wal}). In view of this result, we conjecture that:

\begin{conj}
Let $\Omega$ be bounded. If $\Phi$ is a non-automorphic analytic self-map of
$\Omega$ such that $C_{\Phi}$ is complex symmetric on $H^2(\Omega),$ then $\Phi$ has a fixed
point in $\Omega.$
\end{conj}



\end{document}